\documentclass{amsart}

\usepackage{amsmath}
\usepackage{amsfonts}
\usepackage{amssymb}
\usepackage{graphicx}
\usepackage{hyperref}
\usepackage{mathrsfs}
\usepackage{pb-diagram}
\usepackage{epstopdf}
\usepackage{amsmath,amsfonts,amsthm,enumerate,amscd,latexsym,curves,multibox}
\usepackage{bbm}
\usepackage[mathscr]{eucal}
\usepackage{epsfig,epsf,epic}

\def\bZ{\mathbb{Z}}
\def\bQ{\mathbb{Q}}
\def\bR{\mathbb{R}}
\def\bC{\mathbb{C}}
\def\bP{\mathbb{P}}
\def\bL{\mathbb{L}}
\def\Z{\mathbb{Z}}
\def\R{\mathbb{R}}
\def\i{\mathbf{i}}
\def\C{\mathbb{C}}

\newtheorem{thm}{Theorem}
\newtheorem{lem}[thm]{Lemma}

\newtheorem{prop}[thm]{Proposition}
\newtheorem{defn}[thm]{Definition}

\newtheorem{nb}[thm]{Remark}
\newtheorem{conj}[thm]{Conjecture}

\numberwithin{equation}{section}

\begin{document}

\title[Enumerative meaning of mirror maps for toric CY]{Enumerative meaning of mirror maps\\ for toric Calabi-Yau manifolds}
\author[Chan]{Kwokwai Chan}
\address{Department of Mathematics\\ The Chinese University of Hong Kong\\ Shatin \\ Hong Kong}
\email{kwchan@math.cuhk.edu.hk}
\author[Lau]{Siu-Cheong Lau}
\address{Department of Mathematics\\ Harvard University\\ One Oxford Street\\ Cambridge \\ MA 02138\\ USA}
\email{s.lau@math.harvard.edu}
\author[Tseng]{Hsian-Hua Tseng}
\address{Department of Mathematics\\ Ohio State University\\ 100 Math Tower, 231 West 18th Ave. \\ Columbus \\ OH 43210\\ USA}
\email{hhtseng@math.ohio-state.edu}

\begin{abstract}
We prove that the inverse of a mirror map for a toric Calabi-Yau manifold of the form $K_Y$, where $Y$ is a compact toric Fano manifold, can be expressed in terms of generating functions of genus 0 open Gromov-Witten invariants defined by Fukaya-Oh-Ohta-Ono \cite{FOOO10}. Such a relation between mirror maps and disk counting invariants was first conjectured by Gross and Siebert \cite[Conjecture 0.2 and Remark 5.1]{GS11} as part of their program, and was later formulated in terms of Fukaya-Oh-Ohta-Ono's invariants in the toric Calabi-Yau case in \cite[Conjecture 1.1]{CLL12}.
\end{abstract}

\maketitle

\tableofcontents

\section{Introduction}

Let $X$ be an $n$-dimensional toric Calabi-Yau manifold, i.e. a smooth toric variety with trivial canonical line bundle $K_X\simeq \mathcal{O}_X$. Such a manifold is necessarily noncompact. Let $N=\bZ^n$. Then $X=X_\Sigma$ is defined by a fan $\Sigma$ in $N_\bR=N\otimes_\bZ\bR=\bR^n$. Let $v_0,v_1,\ldots,v_{m-1}\in N$ be the primitive generators of the 1-dimensional cones of $\Sigma$. Without loss of generality, we assume that, for $i=0,1,\ldots,m-1$,
$$v_i=(w_i,1)\in N$$
for some $w_i\in\bZ^{n-1}$ and $w_0=0$. Also, following Gross \cite{G01}, we assume that the fan $\Sigma$ has convex support so that $X$ is a crepant resolution of an affine toric variety with Gorenstein canonical singularities.

The Picard number of $X$ is equal to $l:=m-n$. Let $\{p_1,\ldots,p_l\}$ be a nef basis of $H^2(X;\bZ)$ and let $\{\gamma_1,\ldots,\gamma_l\}\subset H_2(X;\bZ)\cong\bZ^l$ be the dual basis. Each $2$-cycle $\gamma_a$ corresponds to an integral relation
$$\sum_{i=0}^{m-1}Q^a_iv_i=0,$$
where $Q^a:=(Q^a_0,Q^a_1,\ldots,Q^a_{m-1})\in\bZ^m$. We equip $X$ with a toric symplectic structure $\omega$ and regard $(X,\omega)$ as a K\"ahler manifold. We also complexify the K\"ahler class by adding a B-field $\mathbf{i}B\in H^2(X,\mathbf{i}\bR)$ and setting $\omega_\bC = \omega + \mathbf{i}B$.

An important class of examples of toric Calabi-Yau manifolds is given by the total spaces of the canonical line bundles $K_Y$ over compact toric Fano manifolds $Y$, e.g. $K_{\bP^2}=\mathcal{O}_{\bP^2}(-3)$.

In \cite{CLL12}, Leung and the first two authors of this paper study local mirror symmetry for a toric Calabi-Yau manifold $X$ from the viewpoint of the SYZ conjecture \cite{SYZ96}. Starting with a special Lagrangian torus fibration (the Gross fibration) on $X$, we construct the {\em SYZ mirror} of $X$ using $T$-duality modified by instanton corrections and wall-crossing, generalizing the constructions of Auroux \cite{Au07, Au09}. The result is given by the following family of noncompact Calabi-Yau manifolds \cite[Theorem 4.37]{CLL12} (see also \cite[Section 7]{AAK12}):
\begin{equation}\label{SYZ}
\check{X}=\left\{(u,v,z_1,\ldots,z_{n-1})\in\bC^2\times(\bC^\times)^{n-1} \mid uv=\sum_{i=0}^{m-1} (1+\delta_i(q))C_iz^{w_i}\right\},
\end{equation}
where
$$\delta_i(q)=\sum_{\alpha\in H_2^{\textrm{eff}}(X,\bZ)\setminus\{0\}}n_{\beta_i+\alpha}q^\alpha$$
is a generating function of disk open Gromov-Witten invariants. Here, $z^w$ denotes the monomial $z_1^{w^1}\ldots z_{n-1}^{w^{n-1}}$ if $w=(w^1,\ldots,w^{n-1})\in\bZ^{n-1}$;
$H_2^{\textrm{eff}}(X,\bZ)$ is the cone of effective classes;
$q^\alpha$ denotes $\exp(-\int_\alpha \omega_\bC)$ and can be expressed in terms of the complexified K\"ahler parameters
\begin{align*}
q_a = \exp\left(-\int_{\gamma_a}\omega_\bC\right),\quad a=1,\ldots,l;
\end{align*}
the coefficients $C_i\in \bC$ ($i=0,\ldots,m-1$) are related to the complexified K\"ahler parameters $q_a$'s by
$$q_a=\prod_{i=0}^{m-1} C_i^{Q^a_i};$$
$\beta_i\in\pi_2(X,L)$ are classes of the basic disks bounded by a Lagrangian torus fiber $L$, and the coefficients $n_{\beta_i+\alpha}=n_{\beta_i+\alpha}(X,L)$ are 1-pointed genus 0 open Gromov-Witten invariants defined by Fukaya-Oh-Ohta-Ono \cite{FOOO10}  (see Subsection~\ref{subsec:opGW} for more precise definitions).

Notice that the SYZ mirror family (\ref{SYZ}) is entirely written in terms of symplectic-geometric data of $X$. Another striking feature is that (\ref{SYZ}) is expected to be inherently written in canonical flat coordinates. This was first conjectured by Gross and Siebert \cite[Conjecture 0.2 and Remark 5.1]{GS11} where they predicted that period integrals of the mirror can be interpreted as counting of tropical disks (instead of holomorphic disks) in the base of an SYZ fibration for a {\em compact} Calabi-Yau manifold; see also \cite[Example 5.2]{GS11a} where Gross and Siebert observed a relation between the so-called {\em slab functions} which appeared in their program and period computations for $K_{\bP^2}$ in \cite{GZ02}. In \cite{CLL12}, a more precise form of this conjecture for toric Calabi-Yau manifolds, which we recall and clarify below, is stated in terms of the genus 0 open Gromov-Witten invariants $n_{\beta_i+\alpha}$.

Let $\Delta\subset N\otimes_\bZ\bR$ be the convex hull of the finite set $\Sigma(1)=\{v_0,v_1,\ldots,v_{m-1}\}\subset N$. Then $\Delta$ is an $(n-1)$-dimensional integral polytope, which can also be viewed as the convex hull of $\{w_0,w_1,\ldots,w_{m-1}\}$ in $\bZ^{n-1}$. Notice that $$\Delta\cap\bZ^{n-1}=\{w_0,w_1,\ldots,w_{m-1}\}$$
since $X$ is smooth. Denote by $\bL(\Delta) \cong \bC^m$ the space of Laurent polynomials $f\in \bC[z_1^{\pm1},\ldots,z_{n-1}^{\pm1}]$ of the form $\sum_{i=0}^{m-1} A_i z^{w_i}$ (i.e. those with Newton polytope $\Delta$). Also let $\bP_\Delta$ be the projective toric variety defined by the normal fan of $\Delta$.

A Laurent polynomial $f\in\bL(\Delta)$ and hence the associated affine hypersurface $Z_f:=\{(z_1,\ldots,z_{n-1})\in(\bC^\times)^{n-1} \mid f(z_1,\ldots,z_{n-1})=0\}$ in $(\bC^\times)^{n-1}$ is called {\em $\Delta$-regular} \cite{B93} if the intersection of the closure $\bar{Z}_f\subset\bP_\Delta$ with every torus orbit $O\subset\bP_\Delta$ is a smooth subvariety of codimension one in $O$. Denote by $\bL_\textrm{reg}(\Delta)$ the space of all $\Delta$-regular Laurent polynomials. The algebraic torus $(\bC^\times)^n$ acts on $\bL_\textrm{reg}(\Delta)$ by
\begin{align*}
(\bC^\times)^n \times \bL_\textrm{reg}(\Delta) & \to \bL_\textrm{reg}(\Delta),\\
(\lambda_0,\lambda_1,\ldots,\lambda_{n-1})\cdot f(z_1,\ldots,z_{n-1}) & = \lambda_0 f(\lambda_1z_1,\ldots,\lambda_{n-1}z_{n-1}).
\end{align*}

Following Batyrev \cite{B93} and Konishi-Minabe \cite{KM10}, we define the complex moduli space $\mathcal{M}_\bC(\check{X})$ of the mirror Calabi-Yau manifold $\check{X}$ to be the GIT quotient of $\bL_\textrm{reg}(\Delta)$ by this action. Since $0$ is inside the interior of $\Delta$, the moduli space $\mathcal{M}_\bC(\check{X})$ is nonempty and has (complex) dimension $l=m-n$ \cite{B93}. Also, as the integral relations among the lattice points $\Delta\cap\bZ^{n-1}=\{w_0,w_1,\ldots,w_{m-1}\}$ are generated by $\sum_{i=1}^{m-1} Q^a_iw_i=0$ where $\sum_{i=0}^{m-1} Q^a_i=0$, in fact we can write down the local coordinates on $\mathcal{M}_\bC(\check{X})$ explicitly as
\begin{align*}
y_a = \prod_{i=0}^{m-1} A_i^{Q^a_i},\quad a=1,\ldots,l.
\end{align*}
The moduli space $\mathcal{M}_\bC(\check{X})$ parametrizes a family of open Calabi-Yau manifolds $\check{\mathfrak{X}}\to \mathcal{M}_\bC(\check{X})$ defined by
$$\check{X}_y = \left\{(u,v,z_1,\ldots,z_{n-1})\in\bC^2\times(\bC^\times)^{n-1} \mid uv = \sum_{i=0}^{m-1} A_i z^{w_i} \right\},$$
where $y_a = \prod_{i=0}^{m-1} A_i^{Q^a_i}$ for $a=1,\ldots,l$. This was the mirror family originally predicted via physical arguments \cite{CKYZ99, HIV00, GZ02}.

Classically, a mirror map
\begin{align*}
\psi:\mathcal{M}_\bC(\check{X}) & \to \mathcal{M}_K(X),\\
y=(y_1,\ldots,y_l) & \mapsto \psi(y)=(q_1(y),\ldots,q_l(y))
\end{align*}
from the complex moduli space $\mathcal{M}_\bC(\check{X})$ of the mirror to the complexified K\"ahler moduli
$$ \mathcal{M}_K(X):= \{ \exp (\omega + \i B): \omega \in K(X), B \in H^2(X,\R) \}$$ (where $K(X)$ denotes the K\"ahler cone) is defined by period integrals
\begin{align*}
q_a(y)=\exp\left(-\int_{\Gamma_a}\check{\Omega}_y\right),\quad a=1,\ldots,l,
\end{align*}
over integral cycles $\Gamma_1,\ldots,\Gamma_l$ which constitute part of an integral basis of the middle homology $H_n(\check{X}_y;\bZ)$. Here, $\check{\Omega}_y$ is the holomorphic volume form
\begin{align*}
\check{\Omega}_y = \textrm{Res}\left(\frac{1}{uv-\sum_{i=0}^{m-1} A_i z^{w_i}}
\frac{dz_1}{z_1}\wedge\cdots\wedge\frac{dz_{n-1}}{z_{n-1}}\wedge du\wedge dv\right)
\end{align*}
on $\check{X}_y$. A mirror map gives a local isomorphism from $\mathcal{M}_\bC(\check{X})$ to $\mathcal{M}_K(X)$ near $y=0$ and $q=0$, and hence provides canonical flat (local) coordinates on $\mathcal{M}_\bC(\check{X})$.

Based on our mirror construction, we define a map in the reverse direction:
\begin{defn}\label{defn:SYZ_map}
We define the {\em SYZ map}
\begin{align*}
\phi:\mathcal{M}_K(X) & \to \mathcal{M}_\bC(\check{X}),\\
q=(q_1,\ldots,q_l) & \mapsto \phi(q)=(y_1(q),\ldots,y_l(q)),
\end{align*}
by
\begin{align*}
y_a(q)=q_a\prod_{i=0}^{m-1}(1+\delta_i(q))^{Q_i^a},\quad a=1,\ldots,l.
\end{align*}
\end{defn}

Then we have the following conjecture:
\begin{conj}[Conjecture 1.1 in \cite{CLL12}]\label{can_coords}
There exist integral cycles $\Gamma_1,\ldots,\Gamma_l$ forming part of an integral basis of the middle homology $H_n(\check{X}_y;\bZ)$ such that
\begin{align*}
q_a=\exp\left(-\int_{\Gamma_a}\check{\Omega}_{\phi(q)}\right)\textrm{ for $a=1,\ldots,l$,}
\end{align*}
where $\phi(q)$ is the SYZ map defined in Definition~\ref{defn:SYZ_map} in terms of generating functions $1+\delta_i(q)$ of the genus 0 open Gromov-Witten invariants $n_{\beta_i+\alpha}$. In other words, the SYZ map coincides with the inverse of a mirror map.
\end{conj}

\begin{nb}
In \cite[Conjecture 1.1]{CLL12}, it was wrongly asserted that the integral cycles $\Gamma_1,\ldots,\Gamma_l$ gave an integral basis of $H_n(\check{X}_y;\bZ)$. The correct conjecture should be as stated above. We are grateful to the referees for pointing this out.
\end{nb}

\begin{nb}
We expect that Conjecture~\ref{can_coords} can be generalized to include the {\em bulk-deformed} genus 0 open Gromov-Witten invariants defined by Fukaya-Oh-Ohta-Ono in \cite{FOOO11}. Namely, we claim that any period integral $\int_\Gamma\check{\Omega}_y$ over an integral cycle $\Gamma\in H_n(\check{X}_y;\bZ)$ can be written in terms of certain generating functions of bulk-deformed genus 0 open Gromov-Witten invariants. Analogous results for $\bP^2$ was obtained by Gross in his work \cite{G10} on tropical geometry and mirror symmetry.
\end{nb}

Conjecture~\ref{can_coords} not only provides an enumerative meaning to the inverse mirror map, but also explains the integrality of the coefficients of its Taylor series expansions which has been observed earlier (see e.g. Zhou \cite{Z10}). This also shows that one can write down, using generating functions of disk open Gromov-Witten invariants, Gross-Siebert's slab functions which satisfy a normalization condition that is necessary to run the Gross-Siebert program (see \cite[Remark 5.1]{GS11} and \cite[Example 5.2]{GS11a}). In \cite[Section 5.3]{CLL12}, evidences of Conjecture~\ref{can_coords} were given for the toric Calabi-Yau surface $K_{\bP^1}$ and the toric Calabi-Yau 3-folds $\mathcal{O}_{\bP^1}(-1)\oplus\mathcal{O}_{\bP^1}(-1)$, $K_{\bP^2}$, $K_{\bP^1\times \bP^1}$. In the joint work \cite{LLW12} of the second author with Leung and Wu, Conjecture~\ref{can_coords} was verified for all toric Calabi-Yau surfaces.

In this paper, we prove the above conjecture for toric Calabi-Yau manifolds of the form $X = K_Y$ in a weaker sense that the $n$-cycles $\Gamma_a$, $a = 1,\ldots,l$, are allowed to have complex coefficients instead of being integral. The precise statement is as follows:
\begin{thm}\label{thm:main}
For a toric Calabi-Yau manifold of the form $K_Y$ where $Y$ is a compact toric Fano manifold,
there exist linearly independent cycles $\Gamma_1,\ldots,\Gamma_l\in H_n(\check{X}_y;\bC)$ such that
\begin{align*}
q_a=\exp\left(-\int_{\Gamma_a}\check{\Omega}_{\phi(q)}\right)\textrm{ for $a=1,\ldots,l$,}
\end{align*}
where $\phi(q)$ is the SYZ map in Definition~\ref{defn:SYZ_map} defined in terms of generating functions of the genus 0 open Gromov-Witten invariants $n_{\beta_i+\alpha}$.
\end{thm}

Our proof, which mainly relies on the formula proved in \cite{C11} and the toric mirror theorem \cite{G98, LLY99} for compact semi-Fano toric manifolds, can be outlined as follows. First, by the main result of \cite{C11}, genus 0 open Gromov-Witten invariants of $X=K_Y$ involved in Conjecture \ref{can_coords} can be equated with certain genus 0 {\em closed} Gromov-Witten invariants of the $\bP^1$-bundle
\begin{align*}
Z:=\bP(K_Y\oplus \mathcal{O}_Y)
\end{align*}
over $Y$. We observe that the closed Gromov-Witten invariants needed here occur in a certain coefficient of the $J$-function $J_Z$ of $Z$. Since $Z$ is semi-Fano (i.e. the anticanonical bundle $K_Z^{-1}$ is numerically effective), the toric mirror theorem of Givental \cite{G98} and Lian-Liu-Yau \cite{LLY99} can be applied and it says that $J_Z$ is equal to the combinatorially and explicitly defined $I$-function $I_Z$, via a toric mirror map.

Here comes another key observation: the toric mirror map, which is defined as certain coefficients of the $I$-function, can be written entirely in terms of a single function which is precisely the reciprocal of the generating function $1+\delta_0(q)$ of genus 0 open Gromov-Witten invariants of $X$. It is known that components of the toric mirror map of $Z$ satisfy certain GKZ-type differential equations. Combining with the previous observations, it is then easy to deduce that inverse of the SYZ map for $X$ gives solutions to the GKZ hypergeometric system associated to $X$. Our main result Theorem~\ref{thm:main} then follows by noting that period integrals give a basis of solutions of the GKZ system.

To prove the stronger version, Conjecture~\ref{can_coords}, we need to show that, after a suitable normalization of $\check{\Omega}_y$, there exist integral cycles $\Gamma_1,\ldots,\Gamma_l\in H_n(\check{X}_y;\bZ)$ such that the period integrals $\int_{\Gamma_a}\check{\Omega}_y$ have logarithmic terms of the form $\log y_a + \cdots$. This is closely related to integral structures coming from the central charge formula \cite{H06}. We plan to address this problem in the future; see the last subsection for more discussions.

The rest of this paper is organized as follows. In Section~\ref{sec:opGW_J-function}, we recall the formula equating open and closed Gromov-Witten invariants in \cite{C11} and explicitly compute the generating functions $1+\delta_i(q)$ for $X=K_Y$ using $J$-functions and the toric mirror theorem. In Section~\ref{sec:mirror_SYZ_maps}, we compute the toric mirror map for $Z=\bP(K_Y\oplus \mathcal{O}_Y)$ in terms of the functions $1+\delta_i(q)$. In Section~\ref{sec:GKZ_periods}, we first deduce that components of the inverse of the SYZ map for $X$ are solutions to the GKZ hypergeometric system associated to $X$. Then we prove our main result Theorem~\ref{thm:main} by showing that the period integrals $\int_\Gamma \check{\Omega}_y$ give a basis of solutions of the GKZ systems attached to $X$. We end with some discussions about definitions of mirror maps and ways to enhance Theorem~\ref{thm:main} to Conjecture~\ref{can_coords} in Section~\ref{sec:discussions}.

\section*{Acknowledgment}
We are grateful to Mark Gross and Bernd Siebert for sharing their insights and ideas, and to Lev Borisov, Herb Clemens, Shinobu Hosono, Yukiko Konishi and Satoshi Minabe for enlightening discussions and valuable comments on GKZ systems and period integrals. We thank Conan Leung for encouragement and related collaborations. We would also like to express our deep gratitude to the referees for reading our manuscript very carefully and for many useful comments and suggestions which led to a significant improvement in the exposition of this paper.

Part of this work was done when K. C. was working as a project researcher at the Kavli Institute for the Physics and Mathematics of the Universe (Kavli IPMU), University of Tokyo and visiting IH\'ES. He would like to thank both institutes for hospitality and providing an excellent research environment. S.-C. L. also expresses his deep gratitude to Kavli IPMU for hospitality and providing a very nice research and living environment. The work of K. C. described in this paper was substantially supported by a grant from the Research Grants Council of the Hong Kong Special Administrative Region, China (Project No. CUHK404412). The work of H.-H. T. was supported in part by NSF grant DMS-1047777.

\section{Computing open GW invariants via $J$-functions}\label{sec:opGW_J-function}
In this section, we will establish a formula relating the generating functions $1+\delta_i(q)$ of disk open Gromov-Witten invariants of $X=K_Y$ and the toric mirror map for the $\bP^1$-bundle $Z=\bP(K_Y\oplus\mathcal{O}_Y)$ over $Y$. The main result is Theorem \ref{prop_J-function}.

\subsection{Open Gromov-Witten invariants of $K_Y$}\label{subsec:opGW}
In this subsection we recall the formula computing open Gromov-Witten invariants of $K_Y$ in terms of closed Gromov-Witten invariants, proved in \cite{C11}. To begin with, let us briefly recall the definition of the genus 0 open Gromov-Witten invariants $n_{\beta_i+\alpha}$ following Fukaya-Oh-Ohta-Ono \cite{FOOO10}.

Let $X$ be a toric manifold of complex dimension $n$, equipped with a toric K\"ahler structure $\omega$. Let
$$L\subset X$$
be a Lagrangian torus fiber of the moment map associated to the Hamiltonian $T^n$-action on $(X,\omega)$. Let $\beta\in\pi_2(X,L)$ be a relative homotopy class with Maslov index $\mu(\beta)=2$. Consider the moduli space $\mathcal{M}_1(L,\beta)$ of holomorphic disks in $X$ with boundaries lying in $L$ and one boundary marked point representing the class $\beta$. A compactification of $\mathcal{M}_1(L,\beta)$, $$\mathcal{M}_1(L,\beta)\subset\overline{\mathcal{M}}_1(L,\beta)$$ is given by the moduli space $\overline{\mathcal{M}}_1(L,\beta)$ of stable maps from genus 0 bordered Riemann surfaces $(\Sigma,\partial\Sigma)$ to $(X,L)$ with one boundary marked point and class $\beta$.

It is shown in \cite{FOOO09, FOOO10} that $\overline{\mathcal{M}}_1(L,\beta)$ is a Kuranishi space of virtual (real) dimension $n$. Let $[\overline{\mathcal{M}}_1(L,\beta)]^\textrm{vir}$ be its virtual fundamental cycle. This is an $n$-cycle instead of a chain because $\beta$ is of minimal Maslox index and consequently $\partial\overline{\mathcal{M}}_1(L,\beta)=\emptyset$. The pushforward of this cycle by the evaluation map $ev:\overline{\mathcal{M}}_1(L,\beta)\to L$ at the boundary marked point defines a genus 0 {\em open Gromov-Witten invariant}
$$n_\beta:=ev_*([\overline{\mathcal{M}}_1(L,\beta)]^\textrm{vir})\in H_n(L;\bQ)\cong\bQ.$$
In \cite{FOOO10}, it is shown that the number $n_\beta$ is independent of the choice of perturbations by multi-sections and hence is indeed an invariant of $(X, L, \omega)$.
\begin{nb}
Since the moment map of $X$ does not contain singular fibers, there is no wall-crossing; in other words, the invariants $n_\beta$ remain unchanged when we move the Lagrangian torus fiber $L$.
\end{nb}

Let $v_0,v_1,\ldots,v_{m-1}$ be the primitive generators of the fan $\Sigma$ defining $X$ and let $D_0,D_1,\ldots,D_{m-1}\subset X$ be the associated toric prime divisors (i.e. irreducible torus-invariant hypersurfaces) respectively. Holomorphic disks in $X$ with boundaries in $L$ give an additive basis
$$\beta_0,\beta_1,\ldots,\beta_{m-1}$$
of $\pi_2(X,L)\cong\bZ^m$ such that $\beta_i\cdot D_j=\delta_{ij}$. In \cite{CO03}, Cho and Oh proved that, for $i=0,1,\ldots,m-1$, there exists a unique (up to automorphisms of the domain) holomorphic disk $\varphi:(D^2,\partial D^2)\to (X,L)$ passing through a generic point in $L$ and representing the class $\beta_i$. We call such holomorphic disks the {\em basic disks} and hence their classes $\beta_0,\beta_1,\ldots,\beta_{m-1}$ the {\em basic disk classes}.

Suppose that $X$ is semi-Fano, i.e. the anticanonical divisor $-K_X$ is nef. Then it follows from the results of Cho-Oh \cite{CO03} and Fukaya-Oh-Ohta-Ono \cite{FOOO10} that $n_\beta\neq0$ only when $\beta=\beta_i+\alpha$ for some $i$ and some $\alpha\in H_2^\textrm{eff}(X;\bZ)$ with $c_1(\alpha)=0$ (see e.g. \cite[Lemma 5.1]{C11}).

In what follows we will focus on the case when $X=X_\Sigma$ is the total space of the canonical line bundle $K_Y$ over a compact toric Fano manifold $Y$. Such an $X$ is a (noncompact) Calabi-Yau manifold, i.e. $K_X\cong\mathcal{O}_X$. By our convention, the primitive generators of the fan $\Sigma$ are chosen to be of the form $v_i=(w_i,1)$ for $i=0,1,\ldots,m-1$ and so that $w_0=0$. Without loss of generality, we also require that $w_1,\ldots,w_{m-1}\in N':=\bZ^{n-1}$ form the primitive generators of the fan in $N'_\bR=\bR^{n-1}$ defining $Y$. The toric prime divisor $D_0\subset X$ is then nothing but the zero section $Y\hookrightarrow K_Y$.

Since $D_0$ is the only compact toric prime divisor in $X$, the invariant $n_\beta$ is non-zero only when either $\beta=\beta_i$ for some $1\leq i\leq m-1$ or $\beta=\beta_0+\alpha$ for some $\alpha\in H_2(X;\bZ) = H_2(Y;\bZ)$. By the results of Cho-Oh \cite{CO03} mentioned above, we already know that $n_{\beta_i}=1$ for any $i$. The other invariants $n_{\beta_0+\alpha}$ are computed by the following formula:

\begin{thm}[Theorem 1.1 in \cite{C11}]\label{open_closed}
Consider the $\bP^1$-bundle $Z=\bP(K_Y\oplus\mathcal{O}_Y)\to Y$ over $Y$. Let $h\in H_2(Z;\bZ)$ be the fiber class and $[pt]\in H^{2n}(Z;\bC)$ the Poincar\'e dual of a point. Denote by $\langle[pt]\rangle^Z_{0,1,h+\alpha}$ the 1-point genus zero closed Gromov-Witten invariant of $Z$ with insertion the point class $[pt]$. Then we have the equality
\begin{equation*}
n_{\beta_0+\alpha}=\langle[pt]\rangle^Z_{0,1,h+\alpha}.
\end{equation*}
between open and closed Gromov-Witten invariants.
\end{thm}

\subsection{Computation via $J$-function}
By Theorem \ref{open_closed}, in order to compute the genus 0 open Gromov-Witten invariants appearing in Conjecture \ref{can_coords} for $X=K_Y$, it suffices to compute the genus 0 closed Gromov-Witten invariants:
\begin{equation}\label{closed_GW}
\langle[pt]\rangle^Z_{0,1,h+\alpha}.
\end{equation}
In this subsection, we explain how to do this using the (small) $J$-function and the toric mirror theorem \cite{G98, LLY99}.

The $\bP^1$-bundle $Z$ is a toric manifold defined by the fan $\bar{\Sigma}$ generated by $\Sigma(1)=\{v_0,v_1,\ldots,v_{m-1}\}$ together with the additional ray spanned by $v_m := -v_0$. We have $H_2(Z;\bZ) = \bZ \oplus H_2(Y;\bZ)$ and $H^2(Z;\bZ) = \bZ \oplus H^2(Y;\bZ)$. Choose a nef basis $\{p_1,\ldots,p_l,p_{l+1}\}$ (recall that $l=m-n$) of $H^2(Z;\bZ)$ so that $\{p_1,\ldots,p_l\}$ gives a positive basis of $H^2(Y;\bZ)$ with dual basis $\{\gamma_1,\ldots,\gamma_l\}\subset H_2(Y;\bZ)$ and such that $\{\gamma_1,\ldots,\gamma_l,\gamma_{l+1}=h\}$ gives the dual basis in $H_2(Z;\bZ)$. As usual, we denote by $D_0,\ldots,D_{m-1},D_m\subset X$ the toric prime divisors associated to the generators $v_0,\ldots,v_{m-1},v_m$ respectively.

By definition, the {\em small J-function} $J_Z$ of $Z$ is given by \cite{G98}
\begin{equation}\label{J-function}
J_Z(q,z)=e^{\tau/z}\Bigg(1+\sum_\alpha\sum_{d\in H_2^\textrm{eff}(Z;\bZ)\setminus\{0\}}q^d\Big\langle1,\frac{\phi_\alpha}{z-\psi}\Big\rangle^Z_{0,2,d}\phi^\alpha\Bigg),
\end{equation}
where $\tau = \sum_{a=1}^{l+1} p_a \log q_a \in H^2(Z)$, $q^d = \prod_{a=1}^{l+1} q_a^{p_a\cdot d}$ (here $q_a$'s are regarded as formal variables), $\{\phi_\alpha\}\subset H^*(Z)$ is a homogeneous additive basis, and $\{\phi^\alpha\}$ is its dual basis (with respect to the Poincar\'e pairing). Expanding (\ref{J-function}) into a power series in $1/z$ yields
\begin{align*}
J_Z(q,z) & = e^{\tau/z}\Bigg(1+\sum_\alpha\sum_{d\in H_2^\textrm{eff}(Z;\bZ)\setminus\{0\}}q^d\frac{1}{z}\sum_{k\geq0}\langle1,\phi_\alpha\psi^k\rangle^Z_{0,2,d}\frac{\phi^\alpha}{z^k}\Bigg)\\
& = e^{\tau/z}\Bigg(1+\sum_\alpha\sum_{d\in H_2^\textrm{eff}(Z;\bZ)\setminus\{0\}}q^d\frac{1}{z}\sum_{k\geq1}\langle\phi_\alpha\psi^{k-1}\rangle^Z_{0,1,d}\frac{\phi^\alpha}{z^k}\Bigg),
\end{align*}
where in the second equality we use the string equation.

We observe that the closed Gromov-Witten invariants \eqref{closed_GW} we need occur in the coefficient of the $1/z^2$-term of $J_Z$ that takes values in $H^0(Z)$. Indeed, since \eqref{closed_GW} has no descendant insertions, we look at the terms in the above expansion with $k=1$. Furthermore, to get \eqref{closed_GW} we need $\phi_\alpha=[pt]\in H^{2n}(Z)$ and thus $\phi^\alpha=1\in H^0(Z)$.

In order to extract (\ref{closed_GW}) from $J_Z$, we use the explicit formula for $J_Z$ given by the toric mirror theorem. Recall that the {\em $I$-function} $I_Z$ of $Z$ is given by
\begin{equation}
I_Z(y,z)=e^{t/z}\sum_{d\in H_2^\textrm{eff}(Z;\bZ)}y^d\prod_{i=0}^m\frac{\prod_{k=-\infty}^0(D_i+kz)}{\prod_{k=-\infty}^{D_i\cdot d}(D_i+kz)},
\end{equation}
where $t = \sum_{a=1}^{l+1} p_a \log y_a \in H^2(Z)$, $y^d = \prod_{a=1}^{l+1} y_a^{p_a\cdot d}$ (again we regard the $y_a$'s as formal variables) and we identify $D_i$ with its cohomology class in $H^2(Z)$. Here, the product should be expanded into a $1/z$-series by writing $D_i+mz=mz(1+D_i/mz)$. Note that both $J_Z$ and $I_Z$ are $H^*(Z)$-valued formal functions. (In fact, by \cite[Lemma 4.2]{I09}, the $I$-function and hence, via the mirror theorem stated below, the $J$-function are convergent power series near $y=0$ and $q=0$ respectively.)

The $I$-function has the asymptotics
$$I_Z(y,z) = 1 + \frac{\Upsilon(y)}{z} + \textrm{higher order terms in $z^{-1}$},$$
where $\Upsilon(y)$ is a (multi-valued) function with values in $H^2(Z)$. We define the {\em toric mirror map} for $Z$ to be the map $y\mapsto q(y) := \exp \Upsilon(y)$.
The toric mirror theorem applied to the semi-Fano toric manifold $Z$ then says the following
\begin{thm}[Toric mirror theorem \cite{G98,LLY99}]\label{thm:toric_mirror_thm}
The $I$-function and the $J$-function coincides via the toric mirror map $y\mapsto q(y)$, i.e.
\begin{align*}
J(q(y),z) = I(y,z).
\end{align*}
\end{thm}
In view of this, in order to extract the invariants \eqref{closed_GW}, we should look for the coefficient of the $1/z^2$-term of $I_Z$ that takes values in $H^0(Z)$.

Consider the expansion of the factor
\begin{equation}\label{factor_I-function}
\frac{\prod_{k=-\infty}^0(D_i+kz)}{\prod_{k=-\infty}^{D_i\cdot d}(D_i+kz)}
\end{equation}
into a $1/z$-series, achieved by writing $D_i+kz=kz(1+D_i/kz)$. Since we want terms with values in $H^0(Z)$, we cannot have $D_i$ involved. There are three possibilities:
\begin{enumerate}
\item[(1)] $D_i\cdot d>0$: the only term in the expansion of \eqref{factor_I-function} that does not contain $D_i$ is the leading term $\frac{1}{(D_i\cdot d)!}z^{-D_i\cdot d}.$
\item[(2)] $D_i\cdot d=0$: the quotient \eqref{factor_I-function} is just 1 in this case.
\item[(3)] $D_i\cdot d<0$: in this case the quotient \eqref{factor_I-function} is proportional to $D_i$, because of the factor corresponding to $k=0$.
\end{enumerate}
Thus for the terms we need, only cases (1) and (2) can occur.

Note that $-K_Z=\sum_{i=0}^m D_i$. Therefore, in the sum
$$\sum_{d\in H_2^\textrm{eff}(Z;\bZ)\setminus\{0\}}y^d\prod_{i=0}^m \frac{\prod_{k=-\infty}^0(D_i+kz)}{\prod_{k=-\infty}^{D_i\cdot d}(D_i+kz)},$$
the part of the $1/z^2$-term that takes values in $H^0(Z)$ is
\begin{equation}\label{I1}
\sum_d\frac{y^d}{\prod_i(D_i\cdot d)!},
\end{equation}
where the sum is over all $d\in H_2^\textrm{eff}(Z;\bZ)\setminus\{0\}$ such that $c_1(d)=-K_Z\cdot d=2$ and $D_i\cdot d\geq0$ for all $i=0,1,\ldots,m-1,m$. Note that the term in $I_Z$ with $d=0$ and the factor $e^{t/z}$ do not contribute. So the coefficient of the $1/z^2$-term of $I_Z$ that takes values in $H^0(Z)$ is given exactly by \eqref{I1}.

For $d\in H_2^\textrm{eff}(Z;\bZ)\setminus\{0\}$ such that $c_1(d)=-K_Z\cdot d=2$ and $D_i\cdot d\geq0$ for all $i$, there are two possibilities:
\begin{enumerate}
\item[(i)] either $D_j\cdot d=2$ for some $j$ and $D_i\cdot d=0$ for all $i\neq j$, or
\item[(ii)] $D_{j_1}\cdot d=D_{j_2}\cdot d=1$ and $D_i\cdot d=0$ for all $i\not\in\{j_1, j_2\}$.
\end{enumerate}
Case (i) is impossible as $Z$ is compact, so we are left with case (ii). Let $d\in H_2^\textrm{eff}(Z;\bZ)$ be as in case (ii). Note that $H_2^\textrm{eff}(Z;\bZ) = \bZ_{\geq0}\cdot h \oplus H_2^\textrm{eff}(Y;\bZ)$, and $c_1(h)=2$ and $c_1(\alpha)=0$ for any $\alpha\in H_2(Y;\bZ)\subset H_2(Z;\bZ)$. Hence, $d$ must be of the form $d = h + \alpha$ for some $\alpha \in H_2^\textrm{eff}(Y;\bZ)$. Then we have $D_m\cdot d=1$ since $D_m\cdot \alpha=0$. So $d\in H_2(Z;\bZ)$ must be corresponding to an integral relation of the form $v_j+v_m=0$. The only such class is the fiber class $h=\gamma_{l+1}$, and therefore we conclude that the sum \eqref{I1} is simply given by $y_{l+1}$.

\begin{thm}\label{prop_J-function}
For the $\bP^1$-bundle $Z=\bP(K_Y\oplus\mathcal{O}_Y)$ over a compact toric Fano manifold $Y$, we have the following formula:
\begin{equation}
y_{l+1}(q)=q_{l+1}(1+\delta_0(q_1,\ldots,q_l)),
\end{equation}
where $y=y(q)$ is the inverse of the toric mirror map for $Z$ and
\begin{align*}
1+\delta_0(q_1,\ldots,q_l)=\sum_{\alpha\in H_2^\textrm{eff}(Y;\bZ)}n_{\beta_0+\alpha}q^\alpha
\end{align*}
is the generating function for open Gromov-Witten invariants of $X=K_Y$; here $q^\alpha=\prod_{a=1}^l q_a^{p_a\cdot \alpha}$.
\end{thm}
\begin{proof}
By the above calculation, the sum \eqref{I1} for $Z=\bP(K_Y\oplus\mathcal{O}_Y)$ is simply given by $y_{l+1}$. Let $y=y(q)$ be the inverse of the toric mirror map for $Z$. By the toric mirror theorem for semi-Fano toric manifolds (Theorem~\ref{thm:toric_mirror_thm}), we have
$$I_Z(y(q),z)=J_Z(q,z).$$
By comparing the coefficients of the $1/z^2$-terms that takes values in $H^0(Z)$ on both sides, we have
$$y_{l+1}(q_1,\ldots,q_{l+1})=\sum_{d\in H_2^\textrm{eff}(Z;\bZ)\setminus\{0\}}\langle[pt]\rangle^Z_{0,1,d}q^d.$$
By dimension reasons, the invariant $\langle[pt]\rangle^Z_{0,1,d}$ is nonzero only when $c_1(d)=2$; and we know that for $d\in H_2^\textrm{eff}(Z,\bZ)$, $c_1(d)=2$ if and only if $d=h+\alpha$ for some $\alpha\in H_2^\textrm{eff}(Y,\bZ)$. Hence, we have
$$y_{l+1}(q_1,\ldots,q_{l+1})=q_{l+1}\sum_{\alpha\in H_2^\textrm{eff}(Y;\bZ)}\langle[pt]\rangle^Z_{0,1,h+\alpha}q^\alpha,$$
where $q^\alpha=\prod_{a=1}^l q_a^{p_a\cdot\alpha}$ is a monomial in the variables $q_1,\ldots,q_l$.
Note that $H_2^\textrm{eff}(X,\bZ)=H_2^\textrm{eff}(Y,\bZ)$. The proposition now follows from the formula in Theorem~\ref{open_closed}.
\end{proof}

\section{Toric mirror map and SYZ map}\label{sec:mirror_SYZ_maps}

In this section, we show that the inverse of the toric mirror map for the $\bP^1$-bundle $Z$ contains the SYZ map (Definition~\ref{defn:SYZ_map}) for the toric Calabi-Yau manifold $X=K_Y$. The key observation is that both maps are completely determined by one and the same function $1+\delta_0(q)$.

Recall that the toric mirror map for $Z$ is given by the coefficient of the $1/z$-term of the $I$-function $I_Z(y,z)$ that takes values in $H^2(Z)$. By analyzing the quotient \eqref{factor_I-function} as we did in the previous subsection, it is not hard to see that the term (depending on $d\in H_2^\textrm{eff}(Z,\bZ)$)
\begin{equation}\label{eq:term_I-function}
\prod_{i=0}^m \frac{\prod_{k=-\infty}^0(D_i+kz)}{\prod_{k=-\infty}^{D_i\cdot d}(D_i+kz)}
\end{equation}
in $I_Z(y,z)$ will contribute to $H^{\leq2}(Z)$ only when there exists at most one $0\leq j\leq m$ such that $D_j\cdot d < 0$. If $D_i\cdot d\geq0$ for all $i$, then \eqref{eq:term_I-function} is of the form
\begin{align*}
\frac{z^{-c_1(d)}}{\prod_{i=0}^m(D_i\cdot d)!}-\frac{z^{-c_1(d)-1}}{\prod_{i=0}^m(D_i\cdot d)!}\sum_{i:D_i\cdot d>0}\left(\sum_{k=1}^{D_i\cdot d}\frac{1}{k}\right)D_i + \textrm{terms in $H^{>2}(Z)$},
\end{align*}
which does not contribute to the toric mirror map at all since $c_1(d)$ cannot be equal to 1. On the other hand, if there exists $j$ such that $D_j\cdot d<0$ and $D_i\cdot d\geq0$ for all $i\neq j$, then \eqref{eq:term_I-function} is of the form
\begin{align*}
z^{-c_1(d)-1}\frac{(-1)^{-D_j\cdot d-1}(-D_j\cdot d-1)!}{\prod_{i\neq j}(D_i\cdot d)!}D_j + \textrm{terms in $H^{>2}(Z)$}
\end{align*}
which contributes to the toric mirror map whenever $c_1(d)=0$.

This shows that the $I$-function $I_Z(y,z)$ expands as
\begin{align*}
& 1+\left(\sum_{a=1}^{l+1}p_a\log y_a + \sum_{j=0}^m \sum_{\substack{d: c_1(d)=0, \\D_j\cdot d < 0, \\D_i\cdot d\geq0\ \forall i\neq j}} y^d\frac{(-1)^{-D_j\cdot d-1}(-D_j\cdot d-1)!}{\prod_{i\neq j}(D_i\cdot d)!}D_j \right)\frac{1}{z} \\
& + \textrm{higher order terms in $z^{-1}$}
\end{align*}
Writing $D_j = \sum_{a=1}^{l+1} Q^a_jp_a$ ($j=0,1,\ldots,m$), the toric mirror map
$$y=(y_1,\ldots,y_l,y_{l+1}) \mapsto q(y)=(q_1(y),\ldots,q_l(y),q_{l+1}(y))$$
can then be expressed as
\begin{align*}
\log q_a(y) = \log y_a - \sum_{j=0}^m Q^a_j \Xi_j(y),\quad a=1,\ldots,l,l+1,
\end{align*}
where
\begin{align*}
\Xi_j(y) = \sum_{\substack{d: c_1(d)=0, \\D_j\cdot d < 0, \\D_i\cdot d\geq0\ \forall i\neq j}}y^d\frac{(-1)^{-D_j\cdot d}(-D_j\cdot d-1)!}{\prod_{i\neq j}(D_i\cdot d)!},
\end{align*}
$j=0,1,\ldots,m$.

For the $\bP^1$-bundle $Z=\bP(K_Y\oplus \mathcal{O}_Y)$ over a toric Fano manifold $Y$, the function $\Xi_j(y)$ is nonzero only when $j=0$. This can be seen by applying \cite[Proposition 4.3]{GI11} which says that the function $\Xi_j(y)$ is nonzero if and only if $v_j$ is not a vertex of the fan polytope (recall that the fan polytope is the convex hull of the primitive generators $\Sigma(1)$ of the fan $\Sigma$). Moreover, $\Xi_0(y)$ depends only on the variables $y_1,\ldots,y_l$ since $D_m\cdot d = 0$ for every $d$ with $c_1(d)=0$. Therefore we have
\begin{prop}\label{prop_Z}
The toric mirror map for the $\bP^1$-bundle $Z=\bP(K_Y\oplus\mathcal{O}_Y)$ over a compact toric Fano manifold $Y$ is given by
\begin{equation}
\left\{
\begin{aligned}
q_a & = y_a G(y)^{-Q^a_0},\quad a=1,\ldots,l,\\
q_{l+1} & = y_{l+1} G(y)^{-1},
\end{aligned}\label{mirror_Z}
\right.
\end{equation}
where $G(y) = \exp \Xi_0(y)$ is a function of the variables $y_1,\ldots,y_l$.
\end{prop}
\begin{proof}
Note that $Q^a_m = 0$ for $a=1,\ldots,l$, and $Q^{l+1}_0 = Q^{l+1}_m = 1$, $Q^{l+1}_j = 0$ for $j=1,\ldots,m-1$. So the toric mirror map is of the form as stated.
\end{proof}

Combining Proposition~\ref{prop_Z} with Theorem~\ref{prop_J-function}, we arrive at the following
\begin{thm}\label{thm:inverse_mirror_Z}
The inverse of the toric mirror map for the $\bP^1$-bundle $Z=\bP(K_Y\oplus\mathcal{O}_Y)$ over a compact toric Fano manifold $Y$ is given by
\begin{equation}
\left\{
\begin{aligned}
y_a & = q_a (1+\delta_0(q))^{Q^a_0},\quad a=1,\ldots,l,\\
y_{l+1} & = q_{l+1} (1+\delta_0(q)),
\end{aligned}\label{eq:inverse_mirror_Z}
\right.
\end{equation}
where
\begin{align*}
1+\delta_0(q)=\sum_{\alpha\in H_2^\textrm{eff}(Y;\bZ)}n_{\beta_0+\alpha}q^\alpha
\end{align*}
is the generating function for open Gromov-Witten invariants of $X=K_Y$, and $q^\alpha=\prod_{a=1}^l q_a^{p_a\cdot \alpha}$.

In particular, the first $l$ components of the inverse \eqref{eq:inverse_mirror_Z} of the toric mirror map for $Z$
coincide with the SYZ map (Definition~\ref{defn:SYZ_map}) for $X=K_Y$.
\end{thm}

As an immediate application, we have the equality
\begin{align*}
1+\delta_0(q) = G(y(q)) = \exp \Xi_0(y(q)),
\end{align*}
where $y(q)=(y_1(q),\ldots,y_l(q))$ is part of the inverse mirror map of $Z$. This gives an effective method to compute the genus 0 open Gromov-Witten invariants $n_{\beta_0+\alpha}$.

\section{GKZ systems and period integrals}\label{sec:GKZ_periods}

In this section, we prove that components of the inverse of the SYZ map for the toric Calabi-Yau manifold $X=K_Y$ are solutions to the Gel'fand-Kapranov-Zelevinsky (GKZ) hypergeometric system \cite{GKZ89, GKZ90} associated to $X$. Then by showing that the period integrals $\int_\Gamma \check{\Omega}_y$ give a basis of solutions of the GKZ system, we complete the proof of Theorem~\ref{thm:main}.

\subsection{GKZ hypergeometric systems}
Consider the following exact sequence (the ``fan sequence'') from toric geometry of $Z$:
$$ 0 \to H_2(Z;\Z) \to \Z^{m+1} \to N \to 0 $$
where $\Z^{m+1} \to N$ is mapping the standard basis $\{e_i\}_{i=0}^m$ to $\{v_i\}_{i=0}^m$. For each $d \in H_2(Z;\Z)$, define a differential operator
\begin{equation}
\mathcal{P}_d := \prod_{i:\langle D_i, d \rangle > 0}\prod_{k=0}^{\langle D_i, d \rangle-1}(\mathcal{D}_i-k)
-y_d\prod_{i:\langle D_i, d \rangle<0}\prod_{k=0}^{-\langle D_i, d \rangle-1}(\mathcal{D}_i-k),
\end{equation}
where $\mathcal{D}_i := \sum_{a=1}^{l+1} Q^a_i y_a\frac{\partial}{\partial y_a}$, $i=0,1,\ldots,m$. Givental \cite{G98} observes that coefficients of the $I$-function $I_Z(y,z)$ give solutions to the following system of GKZ-type differential equations:
$$\mathcal{P}_d \Psi = 0,\quad d\in H_2(Z;\Z);$$
see also \cite[Lemma 4.6]{I09}. In particular, components of the toric mirror map \eqref{mirror_Z} for $Z$ are solutions of the above system.

On the other hand, the Gel'fand-Kapranov-Zelevinsky (GKZ) system \cite{GKZ89, GKZ90} of differential equations (also called $A$-hypergeometric system) associated to $X$ (or to $\Sigma(1)=\{v_0,v_1,\ldots,v_{m-1}\}$ with parameter $\beta=0$) is the following system of differential equations on functions $\Phi(A)$ of $A=(A_0,A_1,\ldots,A_{m-1})\in \bC^m$:
\begin{equation}\label{eq:GKZ1}
\sum_{i=0}^{m-1} v_i A_i\frac{\partial}{\partial A_i} \Phi = 0,
\end{equation}
\begin{equation}\label{eq:GKZ2}
\left[\prod_{i:\langle D_i, \alpha \rangle > 0}\left(\frac{\partial}{\partial A_i}\right)^{\langle D_i, \alpha \rangle}
-\prod_{i:\langle D_i, \alpha \rangle < 0}\left(\frac{\partial}{\partial A_i}\right)^{-\langle D_i, \alpha \rangle}\right]\Phi = 0,\quad \alpha\in H_2(X;\Z).
\end{equation}
Note that \eqref{eq:GKZ1} consists of $n$ equations.

\begin{thm}\label{prop:SYZ_GKZ}
By writing $y_a = \prod_{i=0}^{m-1} A_i^{Q^a_i}$ for $a=1,\ldots,l$, then the components
$$q_a = q_a(y_1,\ldots,y_l),\quad a=1,\ldots,l$$
of the inverse of the SYZ map for $X=K_Y$ give solutions to the GKZ hypergeometric system \eqref{eq:GKZ1}, \eqref{eq:GKZ2}.
\end{thm}
\begin{proof}
As $y_a = \prod_{i=0}^{m-1} A_i^{Q^a_i}$, we have
$$A_i\frac{\partial}{\partial A_i} = \sum_{a=1}^{l+1} Q^a_i y_a\frac{\partial}{\partial y_a} = \mathcal{D}_i.$$
So
$$\sum_{i=0}^{m-1} v_i A_i\frac{\partial}{\partial A_i} = \sum_{a=1}^{l+1}\left(\sum_{i=0}^{m-1} Q^a_iv_i\right) y_a\frac{\partial}{\partial y_a}.$$
Since $\sum_{i=0}^m Q^a_i=0$ and $Q^a_m = 0$ for $a=1,\ldots,l$, the equations \eqref{eq:GKZ1} are satisfied.

By induction, we compute that
$$\left(\frac{\partial}{\partial A_i}\right)^Q \Phi = A_i^{-Q} \prod_{k=0}^{Q-1}(\mathcal{D}_i-k)\Phi$$
for $Q\in \bZ_{>0}$.
So, for $\alpha \in H_2(X;\Z)$, the differential operator on the left-hand-side of \eqref{eq:GKZ2} can be written as
\begin{align*}
& \prod_{i:\langle D_i, \alpha \rangle > 0} A_i^{-\langle D_i, \alpha \rangle} \prod_{k=0}^{\langle D_i, \alpha \rangle - 1}(\mathcal{D}_i-k) 
- \prod_{i:\langle D_i, \alpha \rangle < 0} A_i^{\langle D_i, \alpha \rangle} \prod_{k=0}^{-\langle D_i, \alpha \rangle - 1}(\mathcal{D}_i-k)\\
= & \left[\prod_{i:\langle D_i, \alpha \rangle > 0} A_i^{-\langle D_i, \alpha \rangle}\right] \left[\prod_{i:\langle D_i, \alpha \rangle > 0} \prod_{k=0}^{\langle D_i, \alpha \rangle - 1}(\mathcal{D}_i-k) - y_a \prod_{i:\langle D_i, \alpha \rangle < 0}\prod_{k=0}^{-\langle D_i, \alpha \rangle - 1}(\mathcal{D}_i-k)\right]\\
= & \left[\prod_{i:\langle D_i, \alpha \rangle > 0} A_i^{-\langle D_i, \alpha \rangle}\right] \mathcal{P}_\alpha,
\end{align*}
where we have used the fact that $\langle D_m, \alpha \rangle = 0$ for $\alpha \in H_2(X;\Z) \subset H_2(Z;\Z)$ in the last equality. Now Theorem~\ref{thm:inverse_mirror_Z} together with results of Givental \cite{G98} and Iritani \cite[Lemma 4.6]{I09} imply that the equations in \eqref{eq:GKZ2} are also satisfied. 

It follows that the inverse of the SYZ map give solutions to the GKZ system \eqref{eq:GKZ1}, \eqref{eq:GKZ2}.
\end{proof}

\subsection{Period integrals}

By \cite{GKZ90}, the number of linearly independent solutions of the GKZ hypergeometric system \eqref{eq:GKZ1}, \eqref{eq:GKZ2} is equal to the normalized volume $\textrm{Vol}(\Delta)$ of the polytope $\Delta$ (recall that it is the convex hull of $\Sigma(1)=\{v_0,v_1,\ldots,v_{m-1}\}$). Here, ``normalized'' means that the volume of a standard $(n-1)$-simplex is 1. We would like to show that the period integrals $\int_\Gamma \check{\Omega}_y$ provide exactly this number of solutions to the GKZ system.

Recall that the family of mirror Calabi-Yau manifolds $\check{\mathfrak{X}}\to \mathcal{M}_\bC(\check{X})$ is defined by
$$\check{X}_y = \{ (u,v,z_1,\ldots,z_{n-1})\in\bC^2\times(\bC^\times)^{n-1} \mid uv = f_y(z_1,\ldots,z_{n-1}) \},$$
where
$$f_y(z_1,\ldots,z_{n-1})=\sum_{i=0}^{m-1} A_i z^{w_i} \in \bC[z_1^{\pm1},\ldots,z_{n-1}^{\pm1}]$$
is a $\Delta$-regular Laurent polynomial, and $y_a = \prod_{i=0}^{m-1} A_i^{Q^a_i}$ for $a=1,\ldots,l$.

We equip the hypersurface $\check{X}_y$ with the K\"ahler structure pull-back from the ambient space:
$$\check{\omega}_y = -\frac{\mathbf{i}}{2}\left(\sum_{k=1}^{n-1}\frac{dz_k}{z_k}\wedge\frac{d\bar{z}_k}{\bar{z}_k}
+ du\wedge d\bar{u} + dv\wedge d\bar{v}\right)\Bigg|_{\check{X}_y}.$$
Consider the Hamiltonian $S^1$-action on $(\check{X}_y,\check{\omega}_y)$
$$e^{\mathbf{i}\theta}\cdot (u,v,z_1,\ldots,z_{n-1}) \mapsto (e^{\mathbf{i}\theta}u, e^{-\mathbf{i}\theta}v, z_1,\ldots,z_{n-1}),$$
with moment map given by
\begin{align*}
\mu:\check{X}_y \to \bR,\quad (u,v,z_1,\ldots,z_{n-1})\mapsto \frac{1}{2}(|u|^2 - |v|^2).
\end{align*}

On the other hand, we view $\check{X}_y$ as a conic fibration via the map
\begin{equation}
\pi: \check{X}_y \to (\bC^\times)^{n-1},\quad (u,v,z_1,\ldots,z_{n-1})\mapsto (z_1,\ldots,z_{n-1}).
\end{equation}
The discriminant locus of $\pi$ is the affine hypersurface $Z_y:=\{(z_1,\ldots,z_{n-1})\in(\bC^\times)^{n-1} \mid f_y(z_1,\ldots,z_{n-1})=0\}$ in $(\bC^\times)^{n-1}$ defined by $f_y$. We restrict $\pi$ to the level set $\mu^{-1}(0)$ of the moment map, which, by abuse of notation, we still call $$\pi:\mu^{-1}(0)\to (\bC^\times)^{n-1}.$$
This map is a circle fibration where the fiber $\{|u|=|v|\}\subset \bC^2$ degenerates to a point over the hypersurface $Z_y$. Notice that $\pi:\mu^{-1}(0)\to (\bC^\times)^{n-1}$ is also the quotient map since $\mu^{-1}(0)/S^1$ is canonically isomorphic to $(\bC^\times)^{n-1}$.

Now choose a generic $y\in \mathcal{M}_\bC(\check{X})$ so that the function $f_y:(\bC^\times)^{n-1} \to \bC$ is Morse-Smale and convenient in the sense of Kouchnirenko (see \cite[Subsections 3.2 and 3.3]{I09} for more details on how to choose such $y$). Then $f_y$ has precisely $\textrm{Vol}(\Delta)$ critical points on $(\bC^\times)^{n-1}$. We also choose $y$ so that all the critical values are distinct. Using Morse-theoretic arguments, Iritani \cite{I09, I11} shows that the relative homology group $H_{n-1}((\bC^\times)^{n-1},Z_y;\bZ)$ is isomorphic to $\bZ^{\textrm{Vol}(\Delta)}$ and, moreover, an integral basis is given by the {\em Lefschetz thimbles} which are families of vanishing cycles over paths in $\bC$ connecting $0$ to the critical values of $f_y$. (These cycles are in fact Lagrangian and were used by Abouzaid \cite{A09} to study Homological Mirror Symmetry for toric manifolds.)

Let $\Lambda\subset (\bC^\times)^{n-1}$ be a Lefschetz thimble representing an integral cycle in $H_{n-1}((\bC^\times)^{n-1},Z_y;\bZ)$. Topologically, $\Lambda$ is an $(n-1)$-dimensional ball $D^{n-1}$ with boundary $\partial D^{n-1} \subset Z_y = f_y^{-1}(0)$. We define $\Gamma\subset \check{X}_y$ by
$$\Gamma = \pi^{-1}(\Lambda),$$
i.e. $\Gamma$ is the inverse image of $\Lambda$ under the circle fibration $\pi:\mu^{-1}(0)\to (\bC^\times)^{n-1}$. Topologically $\Gamma$ is an $n$-sphere $S^n$ and hence defines an integral $n$-cycle in $H_n(\check{X}_y;\bZ)$. This construction gives a map $H_{n-1}((\bC^\times)^{n-1},Z_y;\bZ) \to H_n(\check{X}_y;\bZ)$ which, as we will see below, is an isomorphism. We remark that a similar construction was suggested by Gross in \cite[Section 4]{G01}.

Consider the standard holomorphic $(n-1)$-form
$$\Omega_0 = \frac{dz_1}{z_1}\wedge\cdots\wedge\frac{dz_{n-1}}{z_{n-1}}$$
on $(\bC^\times)^{n-1}$. Following Konishi-Minabe \cite{KM10}, we regard this as defining the relative cohomology class
$$[(\Omega_0,0)] \in H^{n-1}((\bC^\times)^{n-1}, Z_y).$$

\begin{lem}\label{lem:equality_periods}
We have the equality between period integrals:
\begin{align*}
\int_\Gamma \check{\Omega}_y = \int_\Lambda \Omega_0,
\end{align*}
where $\check{\Omega}_y$ is the holomorphic volume form
\begin{align*}
\check{\Omega}_y = \textrm{Res}\left(\frac{1}{uv-f_y(z_1,\ldots,z_{n-1})}
\frac{dz_1}{z_1}\wedge\cdots\wedge\frac{dz_{n-1}}{z_{n-1}}\wedge du\wedge dv\right)
\end{align*}
on $\check{X}_y$.
\end{lem}
\begin{proof}
On $\check{X}_y\cap (\bC^\times)^{n+1}$, the form $\check{\Omega}_y$ is given by
$$\frac{dz_1}{z_1}\wedge\cdots\wedge\frac{dz_{n-1}}{z_{n-1}}\wedge du,$$
see e.g. \cite[Subsection 4.6.4]{CLL12}
The reduced form
$$(\check{\Omega}_y)_{red} = \iota_{(\partial/\partial\theta)^\sharp}\check{\Omega}_y = \frac{dz_1}{z_1}\wedge\cdots\wedge\frac{dz_{n-1}}{z_{n-1}}$$
descends to the quotient $\mu^{-1}(0)/S^1\cong(\bC^\times)^{n-1}$. The equality now follows from integration along fibers of the map $\pi:\mu^{-1}(0)\to (\bC^\times)^{n-1}$.
\end{proof}

\begin{prop}[cf. \cite{H06} and Corollary A.16 in \cite{KM10}]\label{prop:GKZ_period}
The period integrals
$$\int_\Gamma \check{\Omega}_y,\quad \Gamma\in H_n(\check{X}_y;\bZ),$$
provide a $\C$-basis of solutions to the GKZ hypergeometric system \eqref{eq:GKZ1}, \eqref{eq:GKZ2}.
\end{prop}

\begin{proof}
Batyrev \cite{B93} shows that the period integrals $\int_\Lambda \Omega_0$ are solutions to the GKZ system \eqref{eq:GKZ1}, \eqref{eq:GKZ2}. This gives a map from $H_{n-1}((\C^\times)^{n-1},Z_y;\bC)$ to the space of solutions of the GKZ system. Dually, there is a map from the GKZ $D$-module to $H^{n-1}((\C^\times)^{n-1},Z_y;\bC)$ defined by sending the generator $1$ of the $D$-module to $[\Omega_0]$. It was proven in \cite[Section 6]{S98} that $[\Omega_0]$ and its covariant derivatives (with respect to the Gauss-Manin connection) span $H^{n-1}((\C^\times)^{n-1},Z_y;\bC)$, and hence the dual map is surjective.  Since $H_{n-1}((\bC^\times)^{n-1},Z_y;\bZ)$ has rank $\textrm{Vol}(\Delta)$ (\cite{I09, I11}), which is precisely the number of linearly independent solutions of the GKZ system (\cite{GKZ90}), the above map is indeed an isomorphism. Thus the period integrals $\int_\Lambda \Omega_0$ form a basis of solutions of the GKZ system (this is in fact \cite[Corollary 4.4 (3)]{KM10}), and now the result follows from Lemma~\ref{lem:equality_periods}.
\end{proof}

\begin{nb}
Proposition~\ref{prop:GKZ_period} shows that the above construction of $n$-cycles in $\check{X}_y$ from Lefschetz thimbles in $((\bC^\times)^{n-1},Z_y)$ defines an isomorphism (cf. \cite{G01})
$$H_{n-1}((\bC^\times)^{n-1},Z_y;\bQ) \overset{\cong}{\longrightarrow} H_n(\check{X}_y;\bQ).$$
Also notice that this proposition holds for any toric Calabi-Yau manifold.
\end{nb}

\begin{nb}
For toric Calabi-Yau 3-folds, Proposition~\ref{prop:GKZ_period} was proved by Konishi-Minabe in \cite[Corollary A.16]{KM10} (see also \cite{H06}). To generalize their proof to higher dimensions, one just need to generalize \cite[Proposition A.1]{KM10} (which gives an isomorphism $H^n(\check{X}_y)\cong H^{n-1}((\bC^\times)^{n-1},Z_y)$ sending $\check{\Omega}_y$ to $\Omega_0$), the proof of which should work mutatis mutandis.
\end{nb}

Our main result Theorem~\ref{thm:main} now follows from Theorem~\ref{prop:SYZ_GKZ} and Proposition~\ref{prop:GKZ_period}.

\section{Discussions}\label{sec:discussions}

\subsection{B-models and mirror maps}

In this paper, the mirror B-model for a toric Calabi-Yau $n$-fold $X$ is given by the family of $n$-dimensional noncompact Calabi-Yau manifolds $\check{\mathfrak{X}}\to \mathcal{M}_\bC(\check{X})$ defined by
$$\check{X}_y = \left\{ (u,v,z_1,\ldots,z_{n-1}) \in \bC^2\times(\bC^\times)^{n-1} \mid uv = \sum_{i=0}^{m-1} A_i z^{w_i} \right\},$$
where the parameters $(A_0,A_1,\ldots,A_{m-1})\in \bC^m$ are subject to the constraints
$$y_a = \prod_{i=0}^{m-1} A_i^{Q^a_i},\quad a=1,\ldots,l,$$
and $z^w$ denotes the monomial $z_1^{w^1}\ldots z_{n-1}^{w^{n-1}}$ if $w=(w^1,\ldots,w^{n-1})\in\bZ^{n-1}$. We take this geometry as the mirror B-model because it is the most natural one from the viewpoint of the SYZ conjecture \cite{SYZ96}, as have been shown in \cite{CLL12}.
For this model, a mirror map is defined in terms of period integrals
\begin{equation}\label{eq:period_1}
\int_{\Gamma}\check{\Omega}_y
\end{equation}
of a holomorphic volume form $\check{\Omega}_y \in H^n(\check{X}_y)$ over integral cycles $\Gamma \in H_n(\check{X}_y;\bZ)$.

Another geometry that is usually taken as the mirror B-model in the literature is the family of $(n-2)$-dimensional affine hypersurfaces $\mathfrak{Z}\to \mathcal{M}_\bC(\check{X})$ defined by
$$Z_y = \left\{(z_1,\ldots,z_{n-1}) \in (\bC^\times)^{n-1} \mid \sum_{i=0}^{m-1} A_i z^{w_i} = 0 \right\},$$
where the parameters $(A_0,A_1,\ldots,A_{m-1})\in \bC^m$ are subject to the same constraints as above.
For this model, a mirror map is defined in terms of period integrals
\begin{equation}\label{eq:period_2}
\int_{\Lambda} \Omega_0
\end{equation}
of the standard (relative) holomorphic form $\Omega_0 = \frac{dz_1}{z_1}\wedge\cdots\wedge\frac{dz_{n-1}}{z_{n-1}} \in H^{n-1}((\bC^\times)^{n-1},Z_y)$ over integral relative cycles $\Lambda \in H_{n-1}((\bC^\times)^{n-1},Z_y;\bZ)$.

Both models are commonly used in the Physics literature \cite{CKYZ99, HIV00, GZ02, H06}. In \cite[Section 8]{HIV00}, yet another model was introduced, namely, the family of Landau-Ginzburg models $W_y: \bC\times(\bC^\times)^{n-1} \to \bC$ defined by
$$W_y(z_0,z_1,\ldots,z_{n-1}) = z_0 \sum_{i=0}^{m-1} A_i z^{w_i}$$
for $z_0 \in \bC$ and $(z_1,\ldots,z_{n-1}) \in (\bC^\times)^{n-1}$. Notice that $z_0$ is a $\bC$-valued (instead of $\bC^\times$-valued) variable. In general, we expect that all these models are equivalent in a suitable sense; see \cite[Section 8]{HIV00} for a physical argument.

The period integrals \eqref{eq:period_1}, \eqref{eq:period_2} are solutions to the {\em same} GKZ hypergeometric system \eqref{eq:GKZ1}, \eqref{eq:GKZ2} associated to $X$. In practice, GKZ hypergeometric solutions are often used directly to define the mirror maps. But in fact it is not obvious that this definition coincides with the previous ones, which is precisely why extending Theorem~\ref{thm:main} to Conjecture~\ref{can_coords} is nontrivial; see the subsection below.

Alternatively, one may use coefficients of the equivariant $I$-function to define the {\em toric mirror map} for a toric Calabi-Yau manifold $X$.
This leads to an alternative approach to Theorem~\ref{thm:main} which we sketch as follows. First note that a toric mirror theorem can be established for the {\em non-compact} toric Calabi-Yau manifold $X=K_Y$. In this non-compact setting, closed Gromov-Witten invariants of $X$ are defined by the virtual localization formula. Such a definition coincides with closed Gromov-Witten invariants of $Y$ {\em twisted} by the line bundle $K_Y$ and the inverse equivariant Euler class, in the sense of \cite{CG07}. An equivariant mirror theorem for $X$ can then be proved in a number of ways, for example, by using the main theorem of \cite{CCIT09} or by following the arguments in \cite{G98}.

Theorem~\ref{prop_J-function} in this paper relates the generating function of genus 0 open Gromov-Witten invariants of $X$ to the inverse toric mirror map for $Z=\bP(K_Y\oplus\mathcal{O}_Y)$. One can then try to deduce Theorem~\ref{prop:SYZ_GKZ} by comparing the inverse of the toric mirror map for $X$, as given in the equivariant toric mirror theorem, with the inverse of the toric mirror map for $Z$ and use the fact that the equivariant $I$-function for $X$ satisfies the GKZ hypergeometric system \eqref{eq:GKZ1}, \eqref{eq:GKZ2} associated to $X$ to conclude Theorem~\ref{prop:SYZ_GKZ}.

\subsection{Integral structures and Conjecture~\ref{can_coords}}

Our main result Theorem~\ref{thm:main} is weaker than Conjecture~\ref{can_coords} in that the cycles $\Gamma_1,\ldots,\Gamma_l$ may not be integral cycles in $H_n(\check{X}_y;\bZ)$; instead, we have only proved that they are $\bC$-linear combinations of integral cycles.

To strengthen our result to a proof of Conjecture~\ref{can_coords}, we need to show that, after normalizing the holomorphic volume form $\check{\Omega}_y$ so that $\int_{\Gamma_0}\check{\Omega}_y = 1$ where $\Gamma_0\in H_n(\check{X}_y;\bZ)$ is a monodromy-invariant cycle, there are
integral cycles $\Gamma_1,\ldots,\Gamma_l\in H_n(\check{X}_y;\bZ)$ forming part of an integral basis such that the period integrals $\int_{\Gamma_a}\check{\Omega}_y$ are linear in logarithms and of the form
$$\log y_a + \cdots,$$
i.e. the logarithmic terms all have coefficients equal to 1.

Evidences for this claim have been found by Hosono in \cite[Appendix A]{H06} using explicit calculations and constructions of the integral cycles and explicit. In general, the problem is rather subtle and should be closely related to the integral and rational structures in the mirror B-model; see the works by Horja \cite{H99}, Hosono \cite{H06}, Borisov-Horja \cite{BH06}, Iritani \cite{I09, I11} and Katzarkov-Kontsevich-Pantev \cite{KKP08} on these issues and, in particular, see \cite[Conjectures 2.2 and 6.3]{H06} where Hosono relates this claim to the central charge formula and Kontsevich's Homological Mirror Symmetry conjecture.

To prove the claim in our setting, one may try either to compute the monodromy matrix around the limit point $y=0$ in $\mathcal{\check{X}}_\bC$ (presumably this should be a large complex structure limit point in the Hodge-theoretic sense \cite[Chapter 5]{CK99}), or to construct the integral cycles $\Gamma_1,\ldots,\Gamma_l$ explicitly and then directly evaluate the period integrals $\int_{\Gamma_a}\check{\Omega}_y$ and study the leading terms of their power series expansions. We plan to tackle these issues in a future work.

\bibliographystyle{amsalpha}
\bibliography{geometry}

\end{document}